\documentclass[12pt,oneside]{amsart}

\usepackage{latexsym,amssymb,amsmath,graphicx,hyperref}
\usepackage{subfig}
\usepackage{caption}
\numberwithin{equation}{section}
\usepackage{tikz}
\usetikzlibrary{shapes}

\usepackage{fp}
\usepackage{rotating}
\usetikzlibrary{arrows}
\usepackage{longtable}

\newtheorem{theorem}{Theorem}[section]
\newtheorem{lemma}[theorem]{Lemma}
\newtheorem{corollary}[theorem]{Corollary}

\theoremstyle{definition}
\newtheorem{definition}[theorem]{Definition}

\newtheorem{example}[theorem]{Example}

\DeclareMathOperator{\Ind}{Ind}

\newcommand{\PIVAL}{3.14159265358979323846264338} 
\newcounter{i} 
\newcommand{\Circulant}[2] { 
	\begin{tikzpicture}
	\setcounter{i}{0}
	\whiledo{\value{i}<#1}{ 
		\FPmul\tempA{2}{\thei} 
		\FPdiv\tempB{\PIVAL}{#1} 
		\FPmul\tempC{\tempA}{\tempB} 
		\FPcos\varX{\tempC} 
		\FPsin\varY{\tempC} 
		\stepcounter{i} 
		\FPround\varX{\varX}{3}
		\FPround\varY{\varY}{3}
		\node (\thei) at (\varX,\varY)[place]{ }; 
		\foreach \x in {#2} { 
			\pgfmathparse{mod(\x+\thei,#1)} 
			\let\tempB\pgfmathresult
			\pgfmathparse{mod(\thei-\x,#1)} 
			\let\tempA\pgfmathresult
			\ifthenelse{\lengthtest{\tempA pt < 1 pt}}{\FPadd\tempA{\tempA}{#1}}{}
			\ifthenelse{\lengthtest{\tempB pt < 1 pt}}{\FPadd\tempB{\tempB}{#1}}{}
			\ifthenelse{\lengthtest{\tempA pt > \thei pt}}{}{\ifthenelse{\thei = \tempA}{}{\draw [] (\thei) to (\tempA)}};
			\ifthenelse{\lengthtest{\tempB pt > \thei pt}}{}{\ifthenelse{\thei = \tempB}{}{\draw [] (\thei) to (\tempB)}};
		}
	}
	\end{tikzpicture}
}

\begin{document}

\tikzstyle{place}=[draw,circle,minimum size=0.5mm,inner sep=1pt,outer sep=-1.1pt,fill=black]


\title{Cohen-Macaulay Circulant Graphs}
\thanks{Last updated: \today}
\thanks{Research of first two authors supported in part by NSERC Discovery Grants.
Research of the third author supported by an NSERC USRA}

\author{Kevin N. Vander Meulen}
\address{Department of Mathematics\\
Redeemer University College, Ancaster, ON, L9K 1J4, Canada}
\email{kvanderm@redeemer.ca}

\author{Adam Van Tuyl}
\address{Department of Mathematical Sciences\\
Lakehead University, Thunder Bay, ON, P7B 5E1, Canada}
\email{avantuyl@lakeheadu.ca}

\author{Catriona Watt}
\address{Department of Mathematics\\
Redeemer University College, Ancaster, ON, L9K 1J4, Canada}
\email{cwatt@redeemer.ca}


\begin{abstract}
Let $G$ be the circulant graph $C_n(S)$ with $S \subseteq \{1,2,\ldots,
\lfloor \frac{n}{2} \rfloor\}$, and let $I(G)$ denote its the edge ideal
in the ring $R = k[x_1,\ldots,x_n]$.  We consider the problem of determining
when $G$ is Cohen-Macaulay, i.e, $R/I(G)$ is a Cohen-Macaulay ring.  
Because a Cohen-Macaulay graph $G$ must be well-covered, 
we focus on known families of well-covered circulant graphs of the form 
$C_n(1,2,\ldots,d)$.  
We also characterize which cubic circulant graphs are Cohen-Macaulay. 
We end with the observation that even though the well-covered
property is preserved under lexicographical products of graphs, this
is not true of the Cohen-Macaulay property.
\end{abstract}

\maketitle

\section{Introduction}\label{sec:intro}

Let $G = (V_G,E_G)$ denote a finite simple graph on the vertex set $V_G = \{x_1,\ldots,x_n\}$
with edge set $E_G$.   By identifying the vertices of $G$ with the variables of
the polynomial ring $R = k[x_1,\ldots,x_n]$ (here, $k$ is any field), we can associate
to $G$ the quadratic square-free monomial ideal 
\[I(G) = \langle x_ix_j ~|~ \{x_i,x_j\} \in E_G \rangle \subseteq R\]  
called the {\it edge ideal} of $G$.  Edge ideals were first introduced by 
Villarreal \cite{V}.  During the last couple of years, there has been an interest
in determining which graphs $G$ are {\it Cohen-Macaulay}, that is, determining when the ring 
$R/I(G)$ is a Cohen-Macaulay ring solely from the properties of the graphs.  
Although this problem is probably intractable for arbitrary graphs,
results are known for some families of graphs, e.g., chordal graphs \cite{HHZ} and 
bipartite graphs \cite{HH}.  Readers may also be interested in the recent
survey of Morey and Villarreal
\cite{MV} and the textbook of Herzog and Hibi
\cite{HHbook}, especially Chapter 9. 

Our goal is to identify families of circulant graphs
that are Cohen-Macaulay.  Given an integer $n \geq 1$ and a subset
$S \subseteq \{1,2,\ldots,\lfloor \frac{n}{2} \rfloor \}$, the {\it circulant graph}
$C_n(S)$ is the graph on $n$ vertices $\{x_1,\ldots,x_n\}$ 
such that $\{x_i, x_j\}$ is an edge of $C_n(S)$ if and only if 
$\min\{ \vert i-j \vert, n-\vert i-j \vert \}\in S$.
See, for example, the graph $C_{12}(1,3,4)$ in Figure ~\ref{C12}.
\begin{figure}[h]
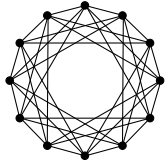

\[\Circulant{12}{1,3,4}\]
\caption{The circulant graph $C_{12}(1,3,4)$.}\label{C12}
\end{figure}
For convenience of notation, we suppress the set brackets
for the set $S=\{1,3,4\}$ in $C_{12}(1,3,4)$. 
Circulant graphs belong to the family of Cayley graphs
and are sometimes viewed as generalized
cycles since $C_n = C_n(1)$.  
The complete graph is also a circulant
graph because  $K_n = C_n(1,2,\ldots,\lfloor \frac{n}{2} \rfloor)$.
In the literature, circulant graphs have appeared
in a number of applications related to networks \cite{BCH}, 
error-correcting codes \cite{ST}, and even music \cite{BHmusic},
in part, because of their regular structure (see \cite{Brown11}).  

To classify families of Cohen-Macaulay circulant 
graphs we will use the fact that all Cohen-Macaulay
graphs must be well-covered.  A graph $G$ is \emph{well-covered} if all the maximal independent
sets of $G$ have the same cardinality, equivalently, every maximal independent set is
a maximum independent set (see the survey
of Plummer \cite{P}). From an algebraic point-of-view, when a graph $G$
is well-covered, the edge ideal is $I(G)$ is unmixed, that
is, all of its associated primes have the same height.  Some families
of well-covered circulant graphs were recently classified by Brown and Hoshino \cite{Brown11}.
Our main results (see Theorems \ref{mainthm} and \ref{cubicmainthm}) refine the work
of Brown and Hoshino by determining which of these well-covered circulant graphs
are also Cohen-Macaulay.   In particular we show in
Theorem~\ref{mainthm} that for $n\geq 2d\geq 2$, the circulant
 $C_n(1,2,\ldots,d)$ is Cohen-Macaulay if and only if $n \leq 3d+2$ and $n \neq 2d+2$.
We also show that the Cohen-Macaulay graphs $C_n(1,2,\ldots,d)$ are 
in fact vertex decomposable and shellable. 
Although 
the well-covered circulant graphs $C_{2d+2}(1,2,\ldots,d)$ and $C_{4d+3}(1,2,\ldots,d)$ are 
not Cohen-Macaulay, we 
prove that these graphs are Buchsbaum (see Theorem~\ref{buchsbaum}).  We also
classify which cubic circulant graphs are Cohen-Macaulay (see 
Theorem \ref{CMcubic}).

Our paper is structured as follows.  In Section 2 we recall the relevant background
regarding graph theory and simplicial complexes.  In Section 3 we classify
the Cohen-Macaulay graphs of the form $C_n(1,2,\ldots,d)$ with $n \geq 2d$.
Section 4 contains the proof of a lemma needed to prove the main result of Section 3.
In Section 5, we look at cubic circulant graphs, and classify those that are 
Cohen-Macaulay.  Section 6
contains some concluding comments and open questions related to the lexicographical
product of graphs.


\section{Background Definitions and Results}\label{sec:background}

A {\it simplicial complex} $\Delta$ on a vertex set $V = \{x_1,\ldots,x_n\}$
is a set of subsets of $V$ that satisfies:
$(i)$ if $F \in \Delta$ and $G \subseteq F$, then $G \in \Delta$, and $(ii)$
for each $i=1,\ldots,n$, $\{x_i\} \in \Delta$.  Note that condition $(i)$ implies that
$\emptyset \in \Delta$.  The elements of $\Delta$ are called
its {\it faces}.  The maximal elements of $\Delta$, with respect to inclusion,
are the {\it facets} of $\Delta$.  

The {\it dimension} of a face $F \in \Delta$
is given by $\dim F = |F|-1$;  the {\it dimension} of a simplicial complex, denoted
$\dim \Delta$, is the maximum dimension of all its faces.  We call $\Delta$
a {\it pure} simplicial complex if all its facets have the same dimension.  Let $f_i$ be
the number of faces of $\Delta$ of dimension $i$, with the convention that $f_{-1} = 1$.
If $\dim \Delta = D$, then the {\it $f$-vector} of $\Delta$ 
is the $(D+2)$-tuple $f(\Delta) = (f_{-1},f_0,f_1,\ldots,f_D)$.  The \emph{$h$-vector} 
of $\Delta$ 
is the $(D+2)$-tuple 
$h(\Delta)=(h_0, h_1, \dots, h_{D+1})$ with (see \cite[Theorem 5.4.6]{Vbook}) 
\[
h_k = \sum_{i = 0}^k (-1)^{k-i}{D+1-i \choose k-i}f_{i-1} \, .
\]

Given any simplicial complex $\Delta$ on $V$, we can associate to $\Delta$
a monomial ideal $I_{\Delta}$ in the polynomial ring $R=k[x_1,\ldots,x_n]$ (with
$k$ a field) as follows:
\[I_{\Delta} = \langle \{x_{j_1}x_{j_2}\cdots x_{j_r} ~|~ \{x_{j_1},\ldots,x_{j_r}\} 
\not\in \Delta\}\rangle.\]
The ideal $I_{\Delta}$ is commonly called the {\it Stanley-Reisner ideal} of $\Delta$,
and the quotient ring $R/I_{\Delta}$ is the {\it Stanley-Reisner ring} of $\Delta$.

We say that $\Delta$ is {\it Cohen-Macaulay} (over $k$) if its Stanley-Reisner
ring $R/I_{\Delta}$ is a Cohen-Macaulay ring,  that is, 
$\operatorname{K-dim}(R/I_{\Delta}) = {\rm depth}(R/I_{\Delta})$.  Here 
$\operatorname{K-dim}(R/I_{\Delta})$, the {\it Krull dimension}, 
is the length of the longest chain of prime ideals
in $R/I_{\Delta}$ with strict inclusions, 
and ${\rm depth}(R/I_{\Delta})$, the
{\it depth}, is length of the longest sequence $f_1,\ldots,f_j$ in $\langle x_1,\ldots,x_n
\rangle$ that forms a regular sequence on $R/I_{\Delta}$.

We review the required background on reduced homology;  see \cite{MS} for complete details.
To any simplicial complex $\Delta$ with $f(\Delta) = (f_{-1},f_0,\ldots,f_D)$ we can associate a
reduced chain complex over $k$:
\[\tilde{C}_{\cdot}(\Delta;k): 0 \longleftarrow k^{f_{-1}} 
\stackrel{\partial_0}{\longleftarrow} k^{f_0}
\stackrel{\partial_1}{\longleftarrow} k^{f_1}
\stackrel{\partial_2}{\longleftarrow} \cdots 
\stackrel{\partial_D}{\longleftarrow}  k^{f_D}
\leftarrow 0 .\]
Here 
$k^{f_i}$ is the vector space with basis elements $e_{j_0,j_1,\ldots,j_i}$ corresponding
to the $i$-dimensional faces of $\Delta$.  We assume $j_0< j_1 < \cdots < j_i$.
 The boundary maps $\partial_i$ are given by
\[\partial_i(e_{j_0,j_1,\ldots,j_i}) = e_{\hat{j}_0,j_1,\ldots,j_i} -
e_{j_0,\hat{j}_1,\ldots,j_i} + 
e_{j_0,j_1,\hat{j}_2\ldots,j_i} + \cdots + (-1)^{i}
e_{j_0,j_1,\ldots,\hat{j}_i}\]
where $~\hat{}~$ denotes an omitted term.  The $i$th {\it reduced simplicial homology} of $\Delta$ 
with coefficients in $k$
is the $k$-vector space 
\[\tilde{H}_i(\Delta;k) = \operatorname{ker}(\partial_i)/\operatorname{im}(\partial_{i+1}).\]
The dimensions of $\tilde{H}_i(\Delta;k)$
are related to $f(\Delta)$  via the {\it reduced Euler
characteristic}:
\begin{equation}\label{euler}
\sum_{i=-1}^D (-1)^i \dim_k \tilde{H}_i(\Delta;k) = \sum_{i=-1}^{D} (-1)^i f_i.
\end{equation}

We will find it convenient to use Reisner's Criterion.
Given a face $F \in \Delta$, the {\it link} of $F$ in $\Delta$ is the
subcomplex
\[{\rm link}_{\Delta}(F) = \{G \in \Delta ~|~ F \cap G = \emptyset,~ F\cup G \in \Delta\}.\]

\begin{theorem}[Reisner's Criterion]\label{reisner} 
Let $\Delta$ be a simplicial complex on $V$.  Then
$R/I_{\Delta}$ is Cohen-Macaulay over $k$ if and only if
for all $F \in \Delta$,
$\tilde{H}_i({\rm link}_{\Delta}(F);k) = 0$ for all $i < \dim
{\rm link}_{\Delta}(F)$.
\end{theorem}

 For any vertex $x \in V$, the {\it deletion} of $x$ in 
$\Delta$ is the subcomplex
\[{\rm del}_\Delta(\{x\}) = \{F \in \Delta ~|~ x \not\in F\}.\]
The following combinatorial topology property 
was 
introduced by Provan and Billera \cite{BP}.

\begin{definition}
Let $\Delta$ be a pure simplicial complex.  Then $\Delta$ is {\it vertex decomposable}
if 
\begin{enumerate}
\item[$(i)$] $\Delta$ is a simplex, i.e $\{x_1,\ldots,x_n\}$ is the unique maximal facet, or
\item[$(ii)$] there exists an  $x \in V$ such that
${\rm link}_{\Delta}(\{x\})$ and ${\rm del}_{\Delta}(\{x\})$ are vertex decomposable.
\end{enumerate}
\end{definition}

We will also refer to the following family of simplicial complexes.

\begin{definition}
Let $\Delta$ be a pure simplicial complex with facets $\{F_1,\ldots,F_t\}$.  Then 
$\Delta$ is {\it shellable} if there exists an ordering of $F_1,\ldots,F_t$ such that
such that for all $1 \leq j < i \leq t$, there is some $x \in F_i \setminus F_j$
and some $k \in \{1,\ldots,j-1\}$ such that $\{x\} = F_i \setminus F_k$.
\end{definition}

The following theorem summarizes a number of necessary and sufficient conditions of
Cohen-Macaulay simplicial complexes that we require. 

\begin{theorem} \label{properties}
Let $\Delta$ be a simplicial complex on the vertex set $V = \{x_1,\ldots,x_n\}$.
\begin{enumerate}
\item[$(i)$] If $\Delta$ is Cohen-Macaulay, then it is pure.
\item[$(ii)$] If $\Delta$ is Cohen-Macaulay, then $h(\Delta)$ has only non-negative
entries.
\item[$(iii)$] If $n - {\rm pdim}(R/I_{\Delta}) = \operatorname{K-dim}(R/I_{\Delta})$, then $\Delta$
is Cohen-Macaulay (here, ${\rm pdim}(R/I_{\Delta})$ denotes the projective dimension
of $R/I_{\Delta}$, the length of a minimal free resolution of $R/I_{\Delta}$).
\item[$(iv)$] If $\Delta$ is vertex decomposable, then $\Delta$ is Cohen-Macaulay.
\item[$(v)$] If $\dim \Delta =0$, then $\Delta$ is vertex decomposable/shellable/Cohen-Macaulay.
\item[$(vi)$] If $\dim \Delta =1$, then $\Delta$ is 
vertex decomposable/shellable/Cohen-Macaulay if and only if $\Delta$ is connected.
\end{enumerate}
\end{theorem}

\begin{proof}  Many of these results are standard.  For $(i)$ see 
\cite[Theorem 5.3.12]{Vbook}; for $(ii)$ see \cite[Theorem 5.4.8]{Vbook}; $(iii)$ follows
from the Auslander-Buchsbaum Theorem; for $(iv)$ see \cite[Corollary 2.9]{BP} and the 
fact that shellable implies Cohen-Macaulay \cite[Theorem 5.3.18]{Vbook}; $(v)$ is 
\cite[Proposition 3.1.1]{BP}; and $(vi)$ is \cite[Theorem 3.1.2]{BP}.
\end{proof}

In this paper, we will be interested in independence complexes of 
finite simple graphs $G = (V_G,E_G)$.   We say that a set of vertices  $W \subseteq V_G$
is an {\it independent set} if for all $e \in E_G$, $e \not\subseteq W$.  The {\it independence
complex} of $G$ 
is the set of all independent sets:
\[{\rm Ind}(G) = \{W ~|~ \mbox{$W$ is an independent set of $V_G$}\}.\]
The set ${\rm Ind}(G)$ is a simplicial complex.  Following
convention, $G$ is Cohen-Macaulay (resp. shellable,
vertex decomposable) if ${\rm Ind}(G)$ is
Cohen-Macaulay (resp. shellable, vertex decomposable).

The facets of ${\rm Ind}(G)$ correspond to the {\it maximal independent sets} of vertices of $G$.
It is common
to let $\alpha(G)$ denote the cardinality of a maximum independent set of vertices in $G$.
A graph $G$ is {\it well-covered} if every maximal independent set has cardinality $\alpha(G)$.
Moreover, a direct translation of the definitions gives us:

\begin{lemma} \label{wellcovered}
If $G$ is Cohen-Macaulay, then
$G$ is well-covered.
\end{lemma}


\section{Characterization of Circulant graphs $C_n(1,2,\ldots,d)$.}\label{characterization}

In this section, we classify which circulant graphs of the form $C_n(1,2,\ldots,d)$ are 
Cohen-Macaulay.
Brown and Hoshino recently classified the well-covered graphs in this family:

\begin{theorem}\label{Brown4.1}
{\rm({\cite[Theorem 4.1]{Brown11}})}. Let $n$ and $d$ be integers with $n \geq 2d \geq 2$. 
Then $C_n(1,2,\ldots,d)$ is well-covered if and only if $n\leq 3d+2$ or $n=4d+3$.
\end{theorem}

Brown and Hoshino's result is a key ingredient for our main result.
We also need one additional result of \cite{Brown11} on the independence
polynomial of $C_n(1,2,\ldots,d)$, but translated into
a statement about $f$-vectors.    The {\it independence polynomial} of a graph $G$
is given by $I(G,x) = \sum_{k=0}^{\alpha(G)} i_kx^k$ where $i_k$ is the number of independent
sets of cardinality $k$ (we take $i_0 = 1$).  Note that if $\Delta = {\rm Ind}(G)$
and $f(\Delta) = (f_{-1},f_0,\ldots,f_D)$, then $i_k = f_{k-1}$ for each $k$.  If
we translate \cite[Theorem 3.1]{Brown11} into the language of $f$-vectors and independence
complexes, we get the following statement.

\begin{lemma}  \label{independencepolynomial}
Let $n$ and $d$ be integers with $n \geq 2d \geq 2$, $G = C_n(1,2,\ldots,d)$, and
$D = \dim {\rm Ind}(G)$.
Then $D = \left\lfloor \frac{n}{d+1} \right\rfloor - 1$ and $f(\Delta) = (f_{-1},f_0,\ldots,f_D)$ where
\[f_{k-1} = \frac{n}{n-dk}\binom{n-dk}{k}~~\mbox{for $k=0,\ldots,(D+1)$.}\]
\end{lemma}
By Lemma \ref{wellcovered}, to characterize the Cohen-Macaulay
circulant graphs of the form $C_n(1,2,\ldots,d)$, it suffices to determine which 
of the well-covered
graphs of Theorem \ref{Brown4.1}
are also Cohen-Macaulay.  Interestingly, proving that
$C_n(1,2,\ldots,d)$ is {\it not} Cohen-Macaulay when $n=4d+3$ is the most subtle part of this
proof.  To carry out this part of the proof, we need the following lemma, whose
proof we postpone until the next section.

\begin{lemma}\label{technicallemma}
Fix an integer $d \geq 3$, and let $G = C_{4d+3}(1,2,\ldots,d)$. If $\Delta = {\rm Ind}(G)$, then
\[\dim_k \tilde{H}_2(\Delta;k) \geq \frac{(4d+3)}{3}\binom{d-1}{2}.\]
\end{lemma}

Assuming, for the moment, that Lemma \ref{technicallemma} holds, we arrive at our main result:

\begin{theorem}\label{mainthm}
Let $n$ and $d$ be integers with $n\geq 2d\geq 2$ and let $G=C_n(1,2,\ldots,d)$. 
Then the following are equivalent:
\begin{enumerate}
\item[$(i)$] $G$ is Cohen-Macaulay.
\item[$(ii)$] $G$ is shellable.
\item[$(iii)$] $G$ is vertex decomposable.
\item[$(iv)$]  $n \leq 3d+2$ and $n \neq 2d+2$.
\end{enumerate}
\end{theorem}

\begin{proof}
We always have $(iii) \Rightarrow (ii) \Rightarrow (i)$.  We now prove that 
$(iv) \Rightarrow (iii)$.

By Lemma \ref{independencepolynomial}, when $n=2d$ or $n=2d+1$,
$\dim {\rm Ind}(G) = 0$.  Now apply Theorem \ref{properties} $(v)$.

When $2d+2\leq n \leq 3d+2$, $\dim {\rm Ind}(G) = 
 \lfloor \frac{n}{d+1} \rfloor - 1 = 1$.  
Let $V = \{x_1,\ldots,x_n\}$.
If $n=2d+2$, then ${\rm Ind}(G)$
is not connected, because the only edges of ${\rm Ind}(G)$
are $\{x_i,x_{d+1+i}\}$ for $i=1,\ldots,d+1$.  On the other hand,
when $2d+3 \leq n \leq 3d+2$, ${\rm Ind}(G)$ is connected.
To see this, let $n =2d+c$ for $3 \leq c \leq d+2$. 
For each $i=1,2,\ldots,n$, $\{x_i,x_{i+d+2}\}$ and 
$\{x_{i+1},x_{i+d+2}\} \in \Ind(G)$, with subscript addition
adjusted modulo $n$.
Thus, for any $x_i,x_j \in V$ with $i < j$, we can make the path
$x_i,x_{i+d+2},x_{i+1},x_{i+d+3},x_{i+2},\ldots,x_j$.  So, ${\rm Ind}(G)$ is connected.  
Applying Theorem \ref{properties} $(vi)$ then shows that $(iv) \Rightarrow (iii)$.

To complete the proof, 
we will show that if $n \geq 2d$ with $n = 2d+2$ or $n > 3d+2$, then $G$
is not Cohen-Macaulay.  In the proof that $(iv) \Rightarrow (iii)$, we already
showed that if $n=2d+2$, then ${\rm Ind}(G)$ is not connected and $\dim {\rm Ind}(G) =1$.  Again
by Theorem \ref{properties} $(vi)$ this implies $G$ is not Cohen-Macaulay. 

If $n > 3d+2$ and $n \neq 4d+3$, then by Theorem \ref{Brown4.1}, $G$ 
is not well-covered, and consequently, by Lemma \ref{wellcovered}, $G$ is not
Cohen-Macaulay.   It therefore remains to show that if $n=4d+3$, 
then $G$ is not Cohen-Macaulay for all $d \geq 1$.  The remainder of this proof
is dedicated to this case.

By Lemma \ref{independencepolynomial}, $\dim {\rm Ind}(G) = 2$ and the
 $f$-vector of ${\rm Ind}(G)$ is given by
\[f({\rm Ind}(G)) = \left(1,4d+3,4d^2+7d+3,\frac{4d^3+15d^2+17d+6}{6}\right).\]
When $d=1$,  then $f(\Ind(G)) = (1,7,14,7)$ and hence $h(\Ind(G))= (1,4,3,-1)$.
When $d=2$, then $f(\Ind(G)) = (1,11,33,21)$ and hence $h(\Ind(G))=(1,8,13,-1)$.
In these two cases,
Theorem \ref{properties} $(ii)$ implies $G$ is not Cohen-Macaulay.

We can therefore assume that $d \geq 3$.  To show that $\Ind(G)$ is not Cohen-Macaulay,
we will show that $\tilde{H}_1(\Ind(G);k) \neq 0$.  This suffices because ${\rm Ind}(G) = {\rm link}_{{\rm Ind}(G)}(\emptyset)$,
so Reisner's Criterion (Theorem \ref{reisner}) would imply that ${\rm Ind}(G)$ is not Cohen-Macaulay.

Using the fact that $\dim \Ind(G) =2$, the $f$-vector given above,
and the reduced Euler characteristic \eqref{euler} 
we know
\[-1+(4d+3)-(4d^2+7d+3)+\frac{4d^3+15d^2+17d+6}{6} = 
\sum_{i=-1}^2 (-1)^i \dim_k \tilde{H}_{i}(\Ind(G);k).\]
Because $\Ind(G)$ is a non-empty connected simplicial complex, we have
$\dim_k \tilde{H}_i(\Ind(G);k) = 0$, for $i = -1$, and $0$.   Simplifying both sides 
of the above equation and rearranging gives:
\[\dim_k\tilde{H}_1(\Ind(G);k) =\dim_k \tilde{H}_2(\Ind(G);k) -   \frac{d(4d^2-9d-1)}{6}.\]
By Lemma \ref{technicallemma}
\[\dim_k \tilde{H}_1(\Ind(G);k) \geq \frac{4d+3}{3}\binom{d-1}{2} - \frac{d(4d^2-9d-1)}{6} = 1.\]
So, $\tilde{H}_1(\Ind(G);k) \neq 0$ as desired.
\end{proof}

When we specialize the above theorem to the case $d=1$, we recover the known classification
of the Cohen-Macaulay cycles~\cite[Corollary 6.3.6]{Vbook}.  Note that $C_2(1)=K_2$ is also Cohen-Macaulay, but
it is not a cycle.

\begin{corollary} Let $n \geq 3$.  Then 
 $C_n = C_n(1)$ is Cohen-Macaulay if and only if $n=3$ or $5$. 
\end{corollary}

Even though $C_{2d+2}(1,2,\ldots,d)$ and $C_{4d+3}(1,2,\ldots,d)$ are 
not Cohen-Macaulay, they still have an interesting
algebraic structure, as noted in Theorem \ref{buchsbaum} below.

\begin{definition}
A pure simplicial complex $\Delta$ is called {\it Buchsbaum} over a field $k$ if 
for every non-empty face $F \in \Delta$,  
$\tilde{H}_i({\rm link}_{\Delta}(F);k) = 0$
for all $i < \dim {\rm link}_{\Delta}(F)$.  We 
say a graph $G$ is {\it Buchsbaum} if the independence complex of $G$ is Buchsbaum.
\end{definition}

Note that by Reisner's Criterion (Theorem \ref{reisner}), if $G$ is Cohen-Macaulay, then
$G$ is Buchsbaum.  We can now classify all circulant graphs of the form 
$C_n(1,\ldots,d)$ which are Buchsbaum, but not Cohen-Macaulay.

\begin{theorem}\label{buchsbaum}
Let $n$ and $d$ be integers with $n\geq 2d$ and $d \geq 1$. Let $G=C_n(1,2,\ldots,d)$.
Then $G$ is Buchsbaum, but not Cohen-Macaulay if and only if $n=2d+2$ or $n=4d+3$.
\end{theorem}

\begin{proof}
$(\Rightarrow)$  For $G$ to be Buchsbaum, ${\rm Ind(G)}$ must be pure, that is
$G$ is well-covered.  By Theorem \ref{Brown4.1}, $2d \leq n \leq 3d+2$ 
or $n=4d+3$.  Because $G$ is not Cohen-Macaulay, Theorem \ref{mainthm}
implies $n=2d+2$ or $n=4d+3$.  

$(\Leftarrow)$  We first show that if $n=4d+3$, then $G$ is Buchsbaum.
Let $\Delta = {\rm Ind}(G)$.  Since $\dim \Delta = 2$, 
given any $F \in \Delta$,  $|F| \in \{0,1,2,3\}$.   We wish to show that if $|F|=1,2$, or $3$, 
then $\tilde{H}_i({\rm link_{\Delta}}(F);k) = 0$
for all $i < \dim {\rm link}_{\Delta}(F)$.

If $|F| = 3$, then ${\rm link}_\Delta(F) = \{\emptyset\}$, and hence  
$\tilde{H}_i({\rm link_{\Delta}}(F);k) = 0$ for all 
$i < \dim {\rm link}_{\Delta}(F) = -1$.  When 
$|F| =2$, then $\dim {\rm link_\Delta(F)} = 0$,  and again, we have  
$\tilde{H}_i({\rm link_{\Delta}}(F);k) = 0$
for all $i < \dim {\rm link}_{\Delta}(F) = 0$.

It therefore suffices to show that when $|F|=1$,  
then $\tilde{H}_i({\rm link_{\Delta}}(F);k) = 0$
for all $i < \dim {\rm link}_{\Delta}(F)$.  Because of the symmetry of $G$,
we can assume without a loss of generality that 
$F = \{x_1\}$.  Because $G$ is well-covered, any independent set containing $x_1$ can be
extended to a maximal independent set, and furthermore, this independent set has cardinality three.
This, in turn, implies that $\dim {\rm link}_{\Delta}(F) = 1$.  For any $i <0$,  
$\tilde{H}_i({\rm link_{\Delta}}(F);k) = 0$, so it suffices to prove that 
$\tilde{H}_0({\rm link_{\Delta}}(F);k) = 0$.  
Proving this condition is equivalent to proving
that ${\rm link}_\Delta(F)$ is connected.

We first note that none of the vertices 
$x_2,x_3,\ldots,x_{d+1},x_{3d+4},x_{3d+5},\ldots,x_{4d+3}$
appear in ${\rm link}_{\Delta}(\{x_1\})$ because these vertices are all adjacent to $x_1$
in $G$.  On the other hand, the following elements are facets of $\Delta$:
\[\begin{array}{l}
\{x_1,x_{d+2},x_{2d+3}\},\{x_1,x_{d+2},x_{2d+4}\},\ldots,\{x_1,x_{d+2},x_{3d+3}\},\\
\{x_1,x_{d+3},x_{3d+3}\},\{x_1,x_{d+4},x_{3d+3}\},\ldots,\{x_1,x_{2d+2},x_{3d+3}\}.
\end{array}\]
Consequently the following edges are in ${\rm link}_{\Delta}(\{x_1\})$:
\small
\[\begin{array}{l}
\{x_{d+2},x_{2d+3}\},\{x_{d+2},x_{2d+4}\},\ldots,\{x_{d+2},x_{3d+3}\},\{x_{d+3},x_{3d+3}\},\{x_{d+4},x_{3d+3}\},\ldots,\{x_{2d+2},x_{3d+3}\}.
\end{array}\]
\normalsize
Thus ${\rm link}_{\Delta}(\{x_1\})$ is connected, as desired.

Now suppose $n=2d+2$. As shown in  the proof of
Theorem \ref{mainthm}, ${\rm Ind}(G)$ consists
of the disjoint edges $\{x_i,x_{d+1+i}\}$ for $i=1,\dots,d+1$.  If $F \in {\rm Ind}(G)$
and $|F|=2$, then ${\rm link}_{\Delta}(F) = \{\emptyset\}$.  If $F \in {\rm Ind}(G)$
and $|F|=1$, then ${\rm link}_{\Delta}(F) = \{\{x\}\}$ for some variable $x$.
Therefore $G$ is Buchsbaum.  
\end{proof}


\section{Proof of Lemma \ref{technicallemma}}

The purpose of this section is to prove Lemma \ref{technicallemma}.
We will be interested in finding induced octahedrons in our independence
complex.  

\begin{lemma}\label{disjointedges}
Fix an integer $d \geq 3$.
Let $G = C_{4d+3}(1,2,\ldots,d)$ and let $\Delta = {\rm Ind}(G)$  
be the associated independence complex.  Let 
$W = \{i_1,i_2,j_1,j_2,k_1,k_2\} \subseteq V_G$ be six distinct vertices.
Then the induced simplicial complex $\Delta|_W = 
\{F \in \Delta ~|~ F \subseteq W\}$ is isomorphic to the 
labeled octahedron in Figure~\ref{octo}
if and only if the induced graph $G_W$ is 
the graph of three disjoint edges $\{i_1,i_2\}$,
$\{j_1,j_2\}$, and $\{k_1,k_2\}$.
\begin{figure}[h]
\[
\begin{tikzpicture}[line join=bevel,z=-5.5]
\coordinate (A1) at (0,0,-1.25);
\coordinate (A2) at (-1.25,0,0);
\coordinate (A3) at (0,0,1.25);
\coordinate (A4) at (1.25,0,0);
\coordinate (B1) at (0,1.25,0);
\coordinate (C1) at (0,-1.25,0);

\draw [fill opacity=0.1, fill=white!80!black] (A1) -- (A2) -- (B1) -- cycle;
\draw [fill opacity=0.3, fill=white!80!black] (A4) -- (A1) -- (B1) -- cycle;
\draw [fill opacity=0.6, fill=white!80!black] (A1) -- (A2) -- (C1) -- cycle;
\draw [fill opacity=0.8, fill=white!80!black] (A4) -- (A1) -- (C1) -- cycle;
\draw [fill opacity=1.0] (A2) -- (A3) -- (B1) -- cycle;
\draw [fill opacity=1.0] (A3) -- (A4) -- (B1) -- cycle;
\draw [fill opacity=1.0] (A2) -- (A3) -- (C1) -- cycle;
\draw [fill opacity=1.0] (A3) -- (A4) -- (C1) -- cycle;

\node [right] at (A1) {$j_2$};
\node [left] at (A2) {$k_1$};
\node [left] at (A3) {$j_1$};
\node [right] at (A4) {$k_2$};
\node [above] at (B1) {$i_1$};
\node [below] at (C1) {$i_2$};
\end{tikzpicture}
\]
\caption{A labeled octahedron}\label{octo} 
\end{figure}
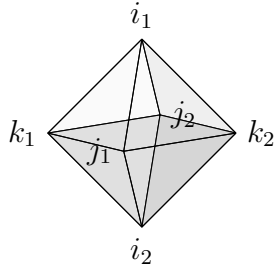
\end{lemma}
\begin{proof}
Suppose that $\Delta|_W$ is isomorphic to the octahedron in Figure~\ref{octo}.
It follows that $\{i_1,i_2\}$, $\{j_1,j_2\}$ and $\{k_1,k_2\}$, which
are not edges of the octahedron, are also not edges of $\Delta$.  Because
$\Delta$ is an independence complex, these means that each set is not
an independent set, or in other words, $e_1 = \{i_1,i_2\}$, $e_2 = \{j_1,j_2\}$, and 
$e_3 = \{k_1,k_2\}$
are all edges of $G$.  It suffices to show that $G_W$ consists only of these edges.
If not, there is a vertex $x \in e_i$ and a vertex $y \in e_j$ with $i \neq j$,
such that $\{x,y\}$ is an edge of $G$.  However, for any $x \in e_i$
and $y \in e_j$, $\{x,y\}$ is an edge of $\Delta|_W$,
and consequently, $\{x,y\}$ cannot be an edge of $G$, a contradiction.

For the converse, we reverse the argument.  If $G_W$ is
the three disjoint edges $\{i_1,i_2\}$, $\{j_1,j_2\}$ and $\{k_1,k_2\}$, then
it follows that $\{i_1,j_1,k_1\},$ $\{i_1,j_1,k_2\},$
$\{i_1,j_2,k_1\},$
$\{i_1,j_2,k_2\},$
$\{i_2,j_1,k_1\},$ 
$\{i_2,j_1,k_2\},$
$\{i_2,j_2,k_1\},$
$\{i_2,j_2,k_2\}$ are all independent sets of $G$, and thus belong to
$\Delta$, and consequently, $\Delta|_W$.    Because $\{i_1,i_2\}$,
$\{j_1,j_2\}$, and $\{k_1,k_2\}$ are not faces of $\Delta$,  the facets
of the complex $\Delta|_W$ are these eight faces, whence $\Delta|_W$ is
an octahedron.
\end{proof}

We now come to our desired proof.

\begin{proof}(of Lemma \ref{technicallemma})
We begin by first recalling some facts about $\Delta = {\rm Ind}(G)$.  
By Theorem \ref{Brown4.1} and Lemma \ref{independencepolynomial}, the simplicial complex
$\Delta$ is pure and two dimensional
with $f(\Delta) = (f_{-1},f_0,f_1,f_2)$.  Therefore, the reduced
chain complex of $\Delta$ over $k$ has the form
\[0 \longleftarrow k^{f_{-1}} \stackrel{\partial_0}{\longleftarrow} 
 k^{f_{0}} \stackrel{\partial_1}{\longleftarrow} 
 k^{f_{1}} \stackrel{\partial_2}{\longleftarrow} 
 k^{f_{2}} \longleftarrow 0.\]
It follows from this chain complex that $\dim_k \tilde{H}_2(\Delta;k)
= \dim_k \ker \partial_2$.

Our strategy, therefore, is to identify  $\frac{(4d+3)}{3}\binom{d-1}{2}$ linearly
independent elements in $\ker \partial_2$.  Note that if $W \subseteq V$
is a subset of the vertices such that the induced complex
$\Delta|_W$ is isomorphic to an 
octahedron, then this octahedron corresponds to an element of $\ker \partial_2$.
We make this more precise.  Suppose that 
$W = \{i_1,i_2,j_1,j_2,k_1,k_2\} \subseteq
V$ and $\Delta|_W$ is an octahedron, i.e, the simplicial complex with facets
\begin{eqnarray*}
\Delta|_W &=& \langle\{i_1,j_1,k_1\},
\{i_1,j_1,k_2\},
\{i_1,j_2,k_1\},
\{i_1,j_2,k_2\}, \\
&& \{i_2,j_1,k_1\}, 
\{i_2,j_1,k_2\},
\{i_2,j_2,k_1\},
\{i_2,j_2,k_2\}\rangle.
\end{eqnarray*}
Note that each $\{i_a,j_b,k_c\}$ is a $2$-dimensional face of $\Delta$; we 
associate to $\Delta|_W$ the following element of $k^{f_2}$:
\[O_W = e_{i_1j_1k_1} - 
 e_{i_1j_1k_2}
- e_{i_1j_2k_1}
- e_{i_2j_1k_1}
+ e_{i_1j_2k_2}
+ e_{i_2j_1k_2}
+ e_{i_2j_2k_1}
- e_{i_2j_2k_2}.\]
Here, we have assumed that the indices of each basis element have been written
in increasing order.  The boundary map $\partial_2$ evaluated at $O_W$ gives
$\partial_2(O_W) = 0$, i.e., $O_W \in \ker \partial_2$.

To compute the lower bound on $\dim_k \tilde{H}_2(\Delta;k)$, we will build a list $L$
of octahedrons in $\Delta$ and then order the elements of $L$ using the lexicographical ordering so
that each octahedron in the list $L$
contains a face that has not appeared in any previous octahedron in 
$L$ with respect to the ordering.  
By associating each octahedron to the corresponding element of $k^{f_2}$,
each octahedron will belong to $\ker \partial_2$.  Moreover, 
the fact that each octahedron in $L$ has a face that has not
appeared previously implies that the octahedron can not be written as
a linear combination of our previous elements in $\ker \partial_2$, thus
giving us the required number of linearly independent elements.   

By Lemma \ref{disjointedges}, there is a one-to-one correspondence between
the induced octahedrons of $\Delta$ and the induced subgraphs of $G$ consisting
of three pairwise disjoint edges.  So we can represent an octahedron by
a tuple $(i_1,i_2;j_1,j_2;k_1,k_2)$ where $\{i_1,i_2\}$, $\{j_1,j_2\}$, and $\{k_1,k_2\}$
correspond to these edges.   

We begin by considering the octahedrons 
described by the following list:
\begin{equation}\label{list1}
\begin{array}{ccc}
(1,2;&d+3,d+4;&2d+5,3d+3) \\
(1,2;&d+3,d+4;&2d+6,3d+3) \\
 &  & \vdots \phantom{, 3d+3} \\
(1,2;&d+3,d+4;&3d+2,3d+3) \\
\hline
(1,2;&d+3,d+5;&2d+6,3d+3) \\
(1,2;&d+3,d+5;&2d+7,3d+3) \\
&  & \vdots \phantom{, 3d+3}\\
(1,2;&d+3,d+5;&3d+2,3d+3) \\
\hline
(1,2;&d+3,d+6;&2d+7,3d+3) \\
& &\vdots 
\phantom{,3d+3} \\
\hline
(1,2;&d+3,2d;\phantom{+2}&3d+1,3d+3) \\
(1,2;&d+3,2d;\phantom{+2}&3d+2,3d+3) \\
\hline
(1,2;&d+3,2d+1;&3d+2,3d+3) 
\end{array}
\end{equation}

If we take our list of 
octahedrons in (\ref{list1}) and add one to each index,
we will get a new list of octahedrons.  In terms of the graph $G = C_{4d+3}(1,2,\ldots,d)$, we are ``rotating''
our disjoint edges to the right.  We ``rotate'' these disjoint edges, or equivalently, we add one
to each index,  until $k_1 = 4d+3$.
So, for example, 
the disjoint edges $(1,2;d+3,d+4;2d+5,3d+3)$ can be rotated to the right $2d-2$
times to give us $2d-1$ octahedrons
\[\begin{array}{ccc}
(1,2;&d+3,d+4;&2d+5,3d+3) \\
(2,3;&d+4,d+5;&2d+6,3d+4) \\
(3,4;&d+5,d+6;&2d+7,3d+5) \\
\vdots&\vdots&\vdots \\
(2d-1,2d;\phantom{+2}&3d+1,3d+2;&4d+3,d-2). 
\end{array}\]
On the other hand, the disjoint edges $(1,2;d+3,2d+1;3d+2,3d+3)$ are only rotated $d+1$ times to 
create $d+2$ octahedrons 
\[\begin{array}{ccc}
(1,2;&d+3,2d+1;&3d+2,3d+3) \\
(2,3;&d+4,2d+2;&3d+3,3d+4) \\
(3,4;&d+5,2d+3;&3d+3,3d+5) \\
\vdots&\vdots&\vdots \\
(d+2,d+3;&2d+4,3d+2;&4d+3,1)\phantom{+5)}. 
\end{array}\]
If we carry out this procedure, we end up with an expanded list $L$ of octahedrons with
\[|L|= \sum_{k=1}^{d-2}k(2d-k) =  2d\sum_{k=1}^{d-2} k - \sum_{k=1}^{d-2} k^2 = \frac{4d+3}{3} \binom{d-1}{2}.  \]
To see why, there is only one collection of disjoint edges with $k_1 = 2d+5$ 
which is rotated $2d-1$ times, there are two tuples of disjoint edges with  $k_1 = 2d+6$
which are rotated $2d-2$ times, and so on, until we arrive at the $d-2$ tuples
which are constructed from all the tuples with $k_1=3d+2$  rotated $d+2$ times.

It now suffices to show that the corresponding
elements of $\ker \partial_2$ are linearly independent. 
In \eqref{list}, we have arranged the list $L$ in lexicographical
order from smallest to largest:
\begin{equation}\label{list}
\begin{array}{ccc}
(1,2;&\phantom{3}d+3,\phantom{3}d+4;&2d+5,3d+3) \\
&\phantom{3d+2;}\vdots&\vdots\phantom{3d+2;} \\
(1,2;&\phantom{3}d+3,2d+1;&3d+2,3d+3) \\
\hline
(2,3;&\phantom{3}d+4,d+5;&2d+6,3d+4) \\
&\phantom{3d+2;}\vdots&\vdots\phantom{3d+2;} \\
(2,3;&\phantom{3}d+4,2d+2;&3d+3,3d+4) \\
\hline
&\vdots& \\
\hline
(2d-2,2d-1;&\phantom{+1}3d,3d+1;&4d+3,d-3) \\
(2d-2,2d-1;&\phantom{+1}3d,3d+2;&4d+3,d-3) \\
\hline
(2d-1,2d;\phantom{-1}&3d+1,3d+2;&4d+3,d-2) \\
\end{array}
\end{equation}
For each $(i_1,i_2;j_1,j_2;k_1,k_2)$ in $L$, 
we consider the two-dimensional face 
$\{i_2,j_2,k_1\}$ 
of the associated octahedron.  
We claim that as we progress down
the list in \eqref{list}, each face $\{i_2,j_2,k_1\}$
has not appeared in a previous octahedron.

In particular, suppose that $(i_1,i_2;j_1,j_2;k_1,k_2)$ is the $\ell$-th item
in \eqref{list}.   We wish to show that the face $\{i_2,j_2,k_1\}$ has not
appeared in any of the first $\ell-1$ octahedrons in the lexicographically ordered list \eqref{list}.
Suppose, that $(a_1,a_2;b_1,b_2,c_1,c_2)$ appears earlier in the list
and contains the face $\{i_2,j_2,k_1\}$.  
For this face to appear,
$\{a_1,a_2\}$ must contain exactly one of $i_2,j_2,k_1$,
$\{b_1,b_2\}$ must contain exactly one of the remaining two vertices, 
and $\{c_1,c_2\}$, must contain
the remaining vertex of the face.

By the way we listed and constructed our octahedrons,
$i_2 < j_2 < k_1$, $i_1=i_2-1$, and
$j_1 = i_2+d+1$. Further, $k_2 = i_2 + 3d+1$ if $i_2+3d+1 \leq (4d-3)$, 
and $k_2 = i_2+3d+1  -(4d-3)$ if
$i_2+3d+1 > (4d-3)$.
Note that $a_1 \neq i_2,j_2$ or $k_2$ otherwise $a_1 > i_1$, contradicting
the lexicographical ordering.  Also, if $a_2 = j_2$,
then $a_1=j_2-1 \geq i_2 > i_1$, again a contradiction.  The same problem arises
if $a_2 =k_1$.  Thus $\{a_1,a_2\} = \{i_2-1,i_2\} = \{i_1,i_2\}$,
i.e., $(a_1,a_2;b_1,b_2;c_1,c_2) = (i_1,i_2,j_1,b_2,c_1,k_2)$.  

Since $j_2$ and $k_1$ must also appear 
in this tuple, there are only two possibilities:
\[(a_1,a_2;b_1,b_2;c_1,c_2) = (i_1,i_2,j_1,k_1,j_2,k_2) ~~\mbox{or} ~~(i_1,i_2,j_1,j_2,k_1,k_2).\]
But neither of these tuples appear strictly before $(i_1,i_2,j_1,j_2,k_1,k_2)$ 
with respect to our ordering, thus completing the proof.
\end{proof}


\section{Cohen-Macaulay Circulant Cubic Graphs}\label{CubicCM}

Brown and Hoshino~\cite{Brown11} classified which circulant cubic graphs are well-covered. 
Recall that a {\it cubic} graph is a graph in which each vertex has degree $3$. Thus, if $G$ is a circulant cubic graph,
then $G = C_{2n}(a,n)$ for some $1 \leq a < n$.  

There are only a finite number of connected well-covered
circulant cubic graphs:

\begin{theorem}[{\cite[Theorem 4.3]{Brown11}}]\label{Brownscubic}
Let $G$ be a connected circulant cubic graph. 
Then $G$ is well-covered if and only if it is isomorphic to one of the 
following graphs: $C_4(1,2)$, $C_6(1,3)$, $C_6(2,3)$, 
$C_8(1,4)$ or $C_{10}(2,5)$.
\end{theorem}

Using a computer algebra system like \emph{Macaulay2} ~\cite{Mt}, one can simply check
which of these graphs, displayed in Figure~\ref{cubic}, are Cohen-Macaulay.  

\begin{theorem}\label{cubicmainthm}
Let $G$ be a connected circulant cubic graph. 
Then is Cohen-Macaulay if and only if it is isomorphic to $C_4(1,2)$ 
or $C_6(2,3)$.
\end{theorem}

\begin{proof}
By Theorem \ref{Brownscubic}, it suffices to check which of the
graphs $C_4(1,2)$, $C_6(1,3)$, $C_6(2,3)$, $C_8(1,4)$ 
or $C_{10}(2,5)$,
\begin{figure}[h]
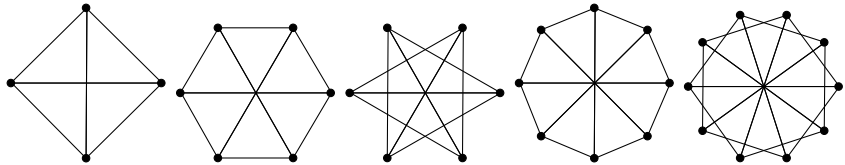

\[\Circulant{4}{1,2}~~\Circulant{6}{1,3}~~\Circulant{6}{2,3}~~\Circulant{8}{1,4}
~~\Circulant{10}{2,5}\] 
\caption{The well-covered connected cubic circulant graphs.}\label{cubic}
\end{figure}
are also Cohen-Macaulay.

For any graph $G$,  $\dim R/I(G) = \alpha(G)$.  So, by
Theorem \ref{properties} $(iii)$, we simply need to check if 
$\alpha(G) = n - \operatorname{pdim}(R/I(G))$.
We can compute $\alpha(G)$ for each of the graphs $G$ in Figure~\ref{cubic} by inspection; on the
other hand, we compute the projective dimension using a computer algebra system.
The following table summarizes these calculations:

\begin{center}
\begin{tabular}{|c|ccccc|}
\hline
\hline
$G$   & $C_4(1,2)$ & $C_6(1,3)$ & $C_6(2,3)$ & $C_8(1,4)$ & $C_{10}(2,5)$ \\
\hline
\hline
$n - \operatorname{pdim}(R/I(G))$ & $1$ & $1$ &$2$ &$2$ &$2$ \\
$\alpha(G)$ &  $1$ & $3$ &$2$& $3$& $4$ \\
\hline
\hline
\end{tabular}
\end{center}
The conclusion now follows from the values in the table.
\end{proof}

As in Brown and Hoshino~\cite{Brown11}, we will use the following result to extend 
Theorem~\ref{cubicmainthm} to all circulant cubic graphs.  The following classification 
is due to Davis and Domke \cite{DD}.

\begin{theorem}\label{3circulant}
Let $G = C_{2n}(a,n)$ with $1 \leq a < n$, and let $t = \gcd(a,2n)$.
\begin{enumerate}
\item[$(i)$] If $\frac{2n}{t}$ is even, then $G$ is isomorphic to $t$ copies of $C_{\frac{2n}{t}}(1,\frac{n}{t})$.
\item[$(ii)$] If $\frac{2n}{t}$ is odd, then $G$ is isomorphic to $\frac{t}{2}$ copies
of $C_{\frac{4n}{t}}(2,\frac{2n}{t})$.
\end{enumerate}
\end{theorem}

We also use the following lemma in the next proof.

\begin{lemma}[{\cite[Proposition 6.2.8]{Vbook}}]\label{disjoint}
Suppose that the graph $G = H \cup K$ where $H$ and $K$ are disjoint components
of $G$.  Then $G$ is Cohen-Macaulay if and only if $H$ and $K$ are Cohen-Macaulay. 
\end{lemma}

\begin{theorem}\label{CMcubic}
Let $G = C_{2n}(a,n)$ with $1 \leq a <n$, that is, $G$ is a cubic circulant graph. Let $t = gcd(a,2n)$.  Then $G$
is Cohen-Macaulay if and only if $\frac{2n}{t} = 3$ or $4$.
\end{theorem}

\begin{proof}
Suppose that $\frac{2n}{t} \neq 3$ or $4$.  If $\frac{2n}{t}$ is even, then 
$C_{\frac{2n}{t}}(1,\frac{n}{t})$ is not Cohen-Macaulay by Theorem \ref{cubicmainthm} and if 
 $\frac{2n}{t}$ is odd, then $C_{\frac{4n}{t}}(2,\frac{2n}{t})$ is also not Cohen-Macaulay
by Theorem \ref{cubicmainthm}.  Thus by Lemma \ref{disjoint} $G$ is not Cohen-Macaulay.
Conversely, if $\frac{2n}{t} = 4$, then by Theorem \ref{3circulant}, $G$ is isomorphic to
$t$ copies of $C_4(1,2)$ and if $\frac{2n}{t} = 3$, then $G$
is isomorphic to $\frac{t}{2}$ copies of $C_6(2,3)$.  In both cases, Theorem
\ref{cubicmainthm} and Lemma \ref{disjoint} imply $G$ is Cohen-Macaulay.   
\end{proof}

\section{Concluding Comments and Open Questions}

The question of classifying {\it all} Cohen-Macaulay circulant graphs
$C_n(S)$ is probably an intractable problem.  Even the weaker question
of determining whether or not a circulant graph $G_n(S)$ is well-covered
(equivalently, ${\rm Ind}(C_n(S))$ is a pure simplicial complex)
was shown by Brown and Hoshino to be co-NP-complete \cite[Theorem 2.5]{Brown11}.
At present, the best we can probably expect is to identify families of
Cohen-Macaulay circulant graphs.

Brown and Hoshino observed that circulant graphs behave well with respect
to the lexicographical product.  Recall this construction:

\begin{definition}
Given two graphs $G$ and $H$, the {\it lexicographical product}, denoted $G[H]$, 
is graph with vertex set $V(G) \times V(H)$, where any two vertices 
$(u,v)$ and $(x,y)$ are adjacent in $G[H]$ if and only if either 
$\{u,x\} \in G$ or $u=x$ and $\{v,y\} \in H$.
\end{definition}

When $G$ and $H$ are both circulant graphs, then the lexicographical
product $G[H]$ is also circulant (see \cite[Theorem 4.6]{Brown11}).  The
well-covered property is also preserved with respect to the lexicographical
product (see \cite{TV}).

\begin{theorem}  \label{lexwellcovered}
Let $G$ and $H$ be two non-empty graphs.  Then $G[H]$ is well-covered
if and only if the graphs $G$ and $H$ are well-covered.
\end{theorem}

As a consequence, the families of well-covered circulant graphs discovered
in \cite{Brown11} can be combined into new well-covered circulant graphs using the
lexicographical product.  
It is therefore natural to ask if the lexicographical product allows us
to build new Cohen-Macaulay circulant graphs from known Cohen-Macaulay circulant
graphs.  In other words, can we replace 
``well-covered'' in Theorem \ref{lexwellcovered} by ``Cohen-Macaulay''.
This turns out not to always be the case, as the following example shows.

\begin{example}
Let $G$ and $H$ be the Cohen-Macaulay circulant graphs $G=C_2(1)$ and 
$H=C_5(1)$. Then $G[H]=C_{10}(1,4,5)$ and $H[G]=C_{10}(1,2,3,5)$ as seen in Figure~\ref{lexgraphs}.
\begin{figure}[h]
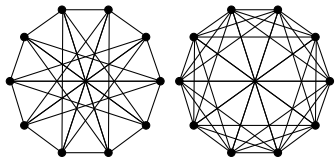

\[\Circulant{10}{1,4,5}~~\Circulant{10}{1,2,3,5}
\]
\caption{Lexicographical products $C_2[C_5]$ and 
$C_5[C_2]$}\label{lexgraphs}
\end{figure}
We can compute $\alpha(G)$ for the graphs in Figure~\ref{lexgraphs} by inspection; on the
other hand, we compute the projective dimension using  {\it Macaulay 2} ~\cite{Mt}.
We find that $\alpha(G[H]) = \dim(R/I(G[H])) = 2 = n-\operatorname{pdim}(R/I(G[H]))$, so
$G[H]$ is Cohen-Macaulay.  However,
$\alpha(H[G]) = \dim (R/I(H[G]))=2 > n - \operatorname{pdim}(R/I(H[G])) = 1$,
so $H[G]$ is not Cohen-Macaulay.
\end{example}

In light of the above example, we can ask what conditions on $G$ and 
$H$ allow us to conclude that the lexicographical product $G[H]$ is Cohen-Macaulay.

We end with a question concerning Lemma \ref{technicallemma}.   Using 
{\it Macaulay 2}~\cite{Mt}, we found that 
$\dim_k \tilde{H}_2(\Delta;k) = \frac{4d+3}{3}\binom{d-1}{2}$
for $d=1,\ldots,14$.  This suggests that the inequality of
Lemma \ref{technicallemma} is actually an equality.  We 
wonder if this is indeed true.

\noindent
{\bf Acknowledgements.}
We thank Brydon Eastman for writing the \LaTeX\ 
code to produce circulant graphs, and Russ Woodroofe and 
Jennifer Biermann for useful discussions.



\end{document}